\documentclass[11pt,reqno]{amsart}
\usepackage{amsmath,amsthm,amsfonts,amssymb,xcolor,enumitem,comment,hyperref}
\usepackage[numbers,sort&compress]{natbib}
\usepackage{graphicx}
\usepackage{setspace}
\usepackage{mathrsfs}
\usepackage{hyperref}
\usepackage[letterpaper,margin=1.3in]{geometry}

\usepackage{comment}

\newcommand{\scott}[1]{{\scriptsize\ \color{red}\textbf{Scott's note:} #1 \color{black}\ \normalsize}}

\renewcommand{\epsilon}{\varepsilon}
\renewcommand{\hat}{\widehat}

\def\XXint#1#2#3{{\setbox0=\hbox{$#1{#2#3}{\int}$}
     \vcenter{\hbox{$#2#3$}}\kern-.5\wd0}}

\newtheorem{theorem}{Theorem}[section]
\newtheorem*{theorem*}{Theorem}
\newtheorem{proposition}[theorem]{Proposition}
\newtheorem{lemma}[theorem]{Lemma}
\newtheorem{corollary}[theorem]{Corollary}

\theoremstyle{definition}
\newtheorem{definition}[theorem]{Definition}

\theoremstyle{remark}
\newtheorem{remark}[theorem]{Remark}

\numberwithin{equation}{section}

\newcommand{\bbG}{\mathbb{G}}
\newcommand{\bbR}{\mathbb{R}}

\title[Directional Pliability, Whitney, and Lusin]{Directional Pliability, Whitney Extension, and Lusin Approximation for Curves in Carnot Groups}

\author[Gareth Speight]{Gareth Speight}
\address{Department of Mathematical Sciences, University of Cincinnati, 2815 Commons Way, Cincinnati, OH 45221, United States}
\email{Gareth.Speight@uc.edu}

\author[Scott Zimmerman]{Scott Zimmerman}
\address{Department of Mathematics, 
The Ohio State University,
100 Math Tower, 231 West 18th Avenue, Columbus, OH 43210, United States}
\email{zimmerman.416@osu.edu}


\allowdisplaybreaks

\begin{document}

\begin{abstract}
We show that, in arbitrary Carnot groups, pliability in a subset of directions is sufficient to guarantee the existence of a Whitney-type extension and a Lusin approximation for curves with tangent vectors in the same set of directions. We apply this to show that every horizontal curve in the Engel group must intersect a $C^{1}$ horizontal curve in a set of positive measure.
\end{abstract}


\date{\today}

\maketitle

\section{Introduction}

Lusin's Theorem shows that a measurable function coincides with a continuous function on its domain up to a set of arbitrarily small measure. More precisely, given a measurable function $f\colon \mathbb{R}\to \mathbb{R}$ and $\varepsilon>0$, there exists a continuous map $F\colon \mathbb{R}\to \mathbb{R}$ such that $$m\{x\in\bbR:F(x)\neq f(x)\}<\varepsilon.$$
This motivates the following question: given a map with particular regularity, does it coincide with a map of higher regularity up to a set of arbitrarily small measure?
Such an approximation is known as a ``Lusin approximation'' owing to Lusin's classical work.
Examples of Lusin approximations with higher notions of regularity include the approximation of absolutely continuous functions by $C^{1}$ functions and of $m$-approximately differentiable functions by $C^m$ functions for $m\geq 1$ \cite{Fed69, LT94, Menne}. One way to quickly construct these Lusin approximations is through the application of a Whitney-type extension theorem. Such a theorem gives hypotheses under which a function defined on some subset of a space can be extended to a smooth map defined on the whole space a la the classical Whitney Extension Theorem \cite{Whi34}. With this extension result in hand, one constructs a Lusin approximation of a nonsmooth map by verifying that the hypotheses of the extension theorem hold in a large enough subset of the domain, then the extension theorem provides the desired Lusin approximating map on the remaining part of the domain.

A large part of geometric analysis and geometric measure theory in Euclidean spaces has recently been studied in more general metric settings. One particularly rich setting is that of Carnot groups \cite{BLU07, Mon02}. Carnot groups are Lie groups whose Lie algebra admits a stratification in which the first layer generates the others via Lie brackets in a particular way (see Section~\ref{s-carnot} for definitions). The number of layers is called the {\em step} of the Carnot group. 
These groups model control theoretic scenarios in which motion through a space is restricted to a low dimensional space of directions.
Points in a Carnot group can be connected by horizontal curves. These are the absolutely continuous curves with tangents in the first layer and model the paths through which a body is allowed to travel. Every Carnot group can be equipped with a natural left-invariant distance, dilations, and a Haar measure which respects the algebraic structure. This all induces a rich geometry on the group and makes the study of analysis and geometry in Carnot groups highly appealing.

The present paper contributes to the study of Whitney extensions and Lusin approximations for horizontal curves in Carnot groups. This was first studied for horizontal curves in the Heisenberg group (a step two Carnot group) in \cite{Spe16}, where the first author showed that every horizontal curve coincides with a $C^{1}$ horizontal curve up to a set of arbitrarily small measure. In the same paper, it was shown that a similar result does not hold in the Engel group (a step three Carnot group). To be more precise, given any $\varepsilon>0$ there exists a horizontal curve $\gamma\colon [0,1]\to \mathbb{E}$ in the Engel group such that, for every $C^{1}$ horizontal curve $\Gamma\colon [0,1]\to \mathbb{E}$, we have
\[m\{t\in [0,1]: \Gamma(t) \neq \gamma(t)\} \geq 1 - \varepsilon,\]
where $m$ denotes Lebesgue measure. Independently, the second author proved in \cite{Zim18} a Whitney extension result for $C^{1}$ horizontal curves in the Heisenberg group. Similar results hold in step two Carnot groups \cite{LS16}. These Whitney extension and Lusin approximation results were then extended by others to a much larger class of Carnot groups \cite{JS16} and sub-Riemannian manifolds \cite{SS18} satisfying a geometric condition known as pliability (Definition \ref{pliability}). 
Roughly speaking, pliability of a Carnot group means that curves that are $C^{1}_{H}$-close to the horizontal line in a horizontal direction have endpoints that form a neighborhood of the endpoint of the horizontal line. Related results study Whitney extensions for curves of higher regularity and a finiteness principle for Whitney extensions in particular Carnot groups \cite{CPS21a, CPS21b, PSZ19, VP06, Zim23}.

A natural question arising from the non-approximation result of \cite{Spe16} mentioned above is whether such a counterexample can be strengthened. Namely, does there exist a horizontal curve in the Engel group that intersects every $C^{1}$ horizontal curve in a set of measure zero? 
In other words, does there exist a horizontal curve $\gamma\colon [0,1]\to \mathbb{E}$ in the Engel group such that, for every $C^{1}$ horizontal curve $\Gamma\colon [0,1]\to \mathbb{E}$,
\[m\{t\in [0,1]: \Gamma(t) \neq \gamma(t)\} = 1?\]

In this paper, we answer this question in the negative. To do this we first show that the results of \cite{JS16} have natural analogues if one assumes pliability in a set of directions rather than pliability in every horizontal direction. In particular, a Whitney extension theorem holds for curve fragments whose candidate derivative takes values in a set of pliable directions. This is our first main result.

\begin{theorem}\label{t-partialwhitney}
    Suppose  $\bbG$ is a Carnot group with horizontal layer $\mathfrak{g}_1$, and fix $\mathfrak{O} \subset \mathfrak{g}_1$.
    If every vector in $\mathfrak{O}$ is pliable, then $\bbG$ has the $C^1_H$ Whitney extension property on $\mathfrak{O}$.
\end{theorem}

Our proof of Theorem \ref{t-partialwhitney} expands upon that of the corresponding result in \cite{JS16}, with additional tracking of which directions need to be pliable. The reverse implication of Theorem \ref{t-partialwhitney} is also likely true and should follow an identical argument as the proof of Theorem~3.8 in \cite{JS16}. More precisely, if $\bbG$ has the $C^1_H$ Whitney extension property on $\mathfrak{O}$, then every vector in $\mathfrak{O}$ should be pliable. However, we did not check the details carefully. 

We next show that every direction in the Engel group is pliable except for those that are a scalar multiple of $X_{2}$ (Theorem \ref{t-engel}). We then use Theorem \ref{t-partialwhitney} and Theorem \ref{t-engel} to prove a partial Lusin approximation result in the Engel group (Theorem \ref{engellusin}). From this, we show that every horizontal curve in the Engel group must intersect some $C^{1}$ horizontal curve in a set of positive measure (Theorem~\ref{answerisno}). This is our second main result.

\begin{theorem}\label{answerisno}
Let $\gamma\colon [a,b]\to \mathbb{R}^{4}$ be a horizonal curve in the Engel group. Then there exists a $C^{1}$ horizontal curve $\Gamma\colon [a,b]\to \mathbb{R}^{4}$ such that
\[m\{t\in [a,b]:\Gamma(t)=\gamma(t)\}>0.\]
\end{theorem}

As a corollary, we may conclude that the hypothesized example of a curve intersecting every $C^1$ horizontal curve in a set of measure 0 cannot exist in the Engel group.

We now describe the organization of the paper. In Section 2 we recall background on Carnot groups, Whitney extension, and pliability. In Section 3 we prove the Whitney extension result for curve fragments whose candidate derivative takes values in a set of pliable directions. In Section 4 we apply these to prove Lusin approximation results in the Engel group and answer the stated question in the negative. Finally in the Appendix we explain the expressions for the horizontal vector fields of the Engel group in second exponential coordinates. These seem well known in the literature, but we were unable to find a concrete justification.

\medskip

\textbf{Acknowledgments:} The authors thank Jonathan Bennett for suggesting the question in the Engel group which led to this paper. 

G. Speight was supported by the National Science Foundation under Award No. 2348715.

\section{Preliminaries}
\label{s-carnot}

\subsection{Carnot Groups}

\subsubsection{Overview}

We recall that a Lie algebra is a vector space $\mathfrak{g}$ equipped with a Lie bracket $[\cdot,\cdot] \colon \mathfrak{g}\times \mathfrak{g}\to \mathfrak{g}$ that is bilinear, satisfies $[x,x]=0$, and satisfies the Jacobi identity $[x,[y,z]]+[y,[z,x]]+[z,[x,y]]=0$. The Lie algebra associated to a Lie group is the space of left invariant vector fields equipped with the Lie bracket $[X,Y]$ defined on smooth functions as $[X,Y](f)=X(Y(f))-Y(X(f))$.

\begin{definition}\label{Carnot}
Consider a Lie algebra $\mathfrak{g}$ that is {\em nilpotent} and {\em stratified},
i.e. it admits a decomposition as a direct sum of subspaces of the form
\[\mathfrak{g}=\mathfrak{g}_{1}\oplus \mathfrak{g}_{2}\oplus \cdots \oplus \mathfrak{g}_{s}\]
such that $\mathfrak{g}_{i}=[\mathfrak{g}_{1},\mathfrak{g}_{i-1}]$ for any $i=2, \ldots, s$, and $[\mathfrak{g}_{1},\mathfrak{g}_{s}]=0$. The subspace $\mathfrak{g}_{1}$ is called the \emph{horizontal layer}. Any connected, simply connected Lie group $\bbG$ whose Lie algebra $\mathfrak{g}$ has such a structure is called a \emph{Carnot group}.
We call $s$ the \emph{step} of $\bbG$.
\end{definition}

\begin{remark}
In any Carnot group $\mathbb{G}$ we adopt the following notation. Let $m_i:=\dim(\mathfrak{g}_i)$, $h_i:=m_1+\dots +m_i$ for $1\leq i\leq s$, and $h_{0}:=0$. A basis $X_1,\dots, X_n$ of $\mathfrak{g}$ is \emph{adapted to the stratification} if $X_{h_{i-1}+1},\dots, X_{h_{i}}$ is a basis of $\mathfrak{g}_i$ for $1\leq i \leq s$. We fix such a basis. Let $r$ be the dimension of the horizontal layer i.e. $r:=m_{1}$.
\end{remark}

\begin{definition}
The map $\exp \colon \mathfrak{g}\to \mathbb{G}$ is defined by $\exp(X)=\gamma(1)$ where $\gamma:\bbR \to \bbG$ is the unique one-parameter subgroup of 
$\bbG$ whose tangent vector at the identity is equal to $X$.
This \emph{exponential map} is a diffeomorphism between $\mathfrak{g}$ and $\mathbb{G}$. 
\end{definition}

\subsubsection{Coordinates}

We can identify $\mathbb{G}$ with $\mathbb{R}^{n}$ using the following diffeomorphism:
\[ (x_{1}, \ldots, x_{n})\in \mathbb{R}^{n} \longleftrightarrow \exp(x_{1}X_{1}+\cdots +x_{n}X_{n})\in \mathbb{G}.\]
This identification is known as \emph{exponential coordinates of the first kind},
and this choice of coordinates induces a Lie group operation on $\bbR^{n}$ according to the Baker-Campbell-Hausdorff formula  \cite[Definition~2.2.11]{BLU07}. This coordinate system will be used throughout Section~\ref{sec-whitney}.

Another possible identification of $\bbG$ with $\bbR^{n}$ is through the mapping
\[(x_{1},\ldots,x_{n}) \in \mathbb{R}^{n} \longleftrightarrow \exp(x_{n}X_{n})\cdots \exp(x_{1}X_{1}) \in \mathbb{G}. \]
This is known as \emph{exponential coordinates of the second kind}, and the map is again a diffeomorphism \cite[1.3.1 Lemma]{FS82}. This coordinate system will be used in Section~\ref{sec-engel}. For an example in one Carnot groups, see the appendix.

In either choice of coordinates, a Haar measure on $\mathbb{G}$ (i.e. a Borel measure that is invariant under left translation by the group operation) can be chosen to be the Lebesgue measure on $\mathbb{R}^n$.

\subsubsection{Structure in exponential coordinates of the first kind}

Let us focus our attention now on exponential coordinates of the first kind. 
The \emph{homogeneity} $d_i\in\mathbb{N}$ of the coordinate $x_i$ is defined by
\[ d_i:=j \quad\text {whenever}\quad h_{j-1}+1\leq i\leq h_{j}.\]
This number tracks the layer of the stratification in which the corresponding basis vector $X_i$ lies.
For any $\lambda >0$, the \emph{dilation} $\delta_\lambda\colon \bbG\to\bbG$, is given in coordinates by
\[\delta_\lambda(x_1, \dots ,x_n)=(\lambda^{d_1}x_1, \dots ,\lambda^{d_n}x_n)\]
and is a group automorphism on $\bbG$.

Since the Lie algebra is nilpotent, one may explicitly compute the group operation in $\mathbb{G}$ from the Lie bracket combinations in $\mathfrak{g}$ via the Baker–Campbell–Hausdorff formula.
The following properties of the group law on $\mathbb{R}^n$ follow from this computation and are well known. See, for example,  \cite[Proposition~2.2.22]{BLU07}.

\begin{proposition}
    \label{p-structure}
    We may write the group law in $\bbG$ as $xy = x+y+Q(x,y)$ for some polynomial $Q=(Q_1,\dots,Q_n)$ where
    \begin{enumerate}
        \item $Q_1=\dots=Q_r=0$,
        \item for $r < i \leq n$, the real valued polynomial $Q_i(x,y)$ may be written as a sum of terms each of which contains a factor of the form $(x_jy_\ell - x_\ell y_j)$ for some $1 \leq \ell < i$ and $1 \leq j < i$,
        \item  $Q_i$ is homogeneous of degree $d_i$ with respect to dilations i.e. $Q_i(\delta_t (x),\delta_t(y)) = t^{d_i} Q_i(x,y)$ for all $x,y \in \mathbb{G}$ and $t > 0$,
        \item if the coordinate $x_i$ has homogeneity $d_i \geq 2$, then $Q_i(x, y)$ depends only on the coordinates of $x$ and $y$ that have homogeneity strictly less than $d_i$.
    \end{enumerate}
\end{proposition}





\subsubsection{Curves and distances in Carnot groups}

Identify $\bbG$ with $\bbR^n$ in either choice of exponential coordinates.
A curve $\gamma \colon [a,b]\to \bbG$ is \emph{absolutely continuous} if it is absolutely continuous as a curve into $\mathbb{R}^{n}$. This implies that the velocity $\dot{\gamma}(t)$ exists for almost every $t\in [a,b]$.
Since the choice of coordinates induces $\bbR^n$ with a Lie group structure with Lie algebra $\mathfrak{g}$, we may view any $X \in \mathfrak{g}$ as a left invariant vector field on $\mathbb{R}^n$. That is, we can write 
$\dot{\gamma}(t)=\sum_{j=1}^{n}u_{j}(t)X_{j}(\gamma(t))$ for some controls $u_{1}, \ldots, u_{n}\in L^{1}[a,b]$.

\begin{definition}
An absolutely continuous curve $\gamma\colon [a,b]\to \bbG$ is \emph{horizontal} if we can write 
$\dot{\gamma}(t)=\sum_{j=1}^{r}u_{j}(t)X_{j}(\gamma(t))$ almost everywhere i.e. if $u_{r+1} = \cdots = u_n = 0$ a.e.
We define the \emph{length} of such a curve by $\ell_{\bbG}(\gamma):=\int_{a}^{b}|u|$. 
\end{definition}

For any $t$ at which $\dot{\gamma}(t)$ exists, we will write $\gamma'(t) = \sum_{j=1}^{r}u_{j}(t)X_{j} \in \mathfrak{g}$. In doing so we view $\gamma'$ as a map into $\mathfrak{g}$.
This gives the following alternate definition of horizontality. 

\begin{definition}
An absolutely continuous curve $\gamma\colon [a,b]\to \bbG$ is \emph{horizontal} if $\gamma' \in \mathfrak{g}_1$ almost everywhere.
\end{definition}



\begin{definition}
The \emph{Carnot-Carath\'{e}odory distance (CC distance)} between any two points $x, y\in \bbG$ is defined by
\[d(x,y):=\inf \{\ell_{\bbG}(\gamma)\colon \gamma \mbox{ is a horizontal curve joining } x \mbox{ and }y\}.\]
\end{definition}


Note that the CC distance is left-invariant and homogeneous with repect to dilations i.e. $d(rp,rq) = d(p,q)$ and $d(\delta_{\lambda}(p),\delta_{\lambda}(q)) = {\lambda} d(p,q)$ for any $p,q,r \in \bbG$ and $\lambda > 0$.


We will also make use of another distance which is useful for computations. In any Carnot group $\mathbb{G}$ represented in first exponential coordinates, 
if we write $x=(\hat{x}_1,\dots,\hat{x}_s) \in \mathbb{R}^{\dim \mathfrak{g}_1} \times \cdots \times \mathbb{R}^{\dim \mathfrak{g}_s}$,
there exist positive constants $\lambda_1,\dots,\lambda_s$ such that the following defines a norm on $\mathbb{G}$ \cite{Cas16}:
    $$
    \Vert x \Vert_{Box} := \max \{ \lambda_i |\hat{x}_k|^{1/k} : k=1,\dots, s \} \qquad \mbox{for }x\in \mathbb{G}.
    $$
Then $d_{Box}(x,y) := \Vert y^{-1}x \Vert_{Box}$ for $x,y\in \mathbb{G}$ defines another distance that is left invariant, homogeneous, and continuous with respect to the Euclidean topology. Hence it is bi-Lipschitz equivalent to the CC-metric $d$ on $\bbG$, namely there is a constant $C>0$ 
such that
\begin{equation}
    \label{e-supmetric}
    C^{-1}d(x,y) \leq d_{Box}(x,y) \leq C d(x,y) \qquad \mbox{for }x,y\in \bbG.
\end{equation}
Note that $C$ is independent of $x$ and $y$.

\subsubsection{Smooth curves and pliability}

\begin{definition}
    Given an interval $I \subset \bbR$, a curve $\gamma:I \to \bbG$ is $C^1_H$ if it is $C^1$ (i.e. continuously differentiable) as a curve in $\bbR^n$ and if it is a horizontal curve. Equivalently, $\gamma$ is $C^1_H$ if $\gamma':I \to \mathfrak{g}_1$ exists and is continuous.
\end{definition}

We denote by $C_{H}^{1}(I,\bbG)$ the set of all curves $C^1_H$ curves $\gamma:I \to \bbG$. 
The fixed basis $X_{1},\cdots, X_{r}$ of $\mathfrak{g}_{1}$ defines a natural inner product norm $\|\cdot \|$ on $\mathfrak{g}_{1}$ with respect to which the basis vectors are orthonormal. The space $C_{H}^{1}([a,b],\bbG)$ can be equipped with a natural metric defined by
\[ d(\gamma_{1},\gamma_{2}):=\max \left( \sup_{t\in [a,b]} d(\gamma_{1}(t),\gamma_{2}(t)), \ \sup_{t\in [a,b]} \|\gamma_{2}'(t)-\gamma_{1}'(t)\| \right).\]
We will sometimes use the notation $\|\gamma_{2}'-\gamma_{1}'\|_{\infty}$ to denote $\sup_{t\in [a,b]} \|\gamma_{2}'(t)-\gamma_{1}'(t)\|$. If working in coordinates, we will occasionally use $|\cdot|$ to denote the Euclidean norm on $\mathbb{G}$. 

The following geometric property of curves was introduced in \cite{JS16}.

\begin{definition}\label{pliability}
A curve $\gamma \in C_{H}^{1}([a,b],\bbG)$ is \emph{pliable} if for every neighborhood $\mathcal{V}$ of $\gamma$ in $C_{H}^{1}([a,b],\bbG)$ the set
\[ \{ (\beta(b),\beta'(b)): \beta \in \mathcal{V}, (\beta,\beta')(a)=(\gamma,\gamma')(a)                        \} \]
is a neighborhood of $(\gamma(b),\gamma'(b))$ in $\bbG\times \mathfrak{g}_{1}$.

A vector $V\in \mathfrak{g}_{1}$ is pliable if the curve $t\mapsto \exp(tV)$ for $t \in [0,1]$ is pliable. The group $\bbG$ is pliable if all vectors $V \in \mathfrak{g}_1$ are pliable.
\end{definition}

\subsection{The Whitney Extension Property}

Recall that a map $f\colon \mathbb{G}_{1}\to \mathbb{G}_{2}$ between Carnot groups is \emph{Pansu differentiable} \cite{DMM25, Mon01, Pan89} at a point $x\in \bbG_{1}$ if there exists a \emph{group linear map} $L\colon \bbG_{1}\to \bbG_{2}$ such that
\begin{equation}
\label{e-pansuDQ}
\lim_{y\to x}\frac{d_2(f(y),f(x)L(x^{-1}y))}{d_1(x,y)}=0
\end{equation}
where $d_1$ and $d_2$ are the CC-metrics in $\bbG_1$ and $\bbG_2$ respectively.
Here a group linear map is one which respects the group operation and dilations. The following is Pansu's generalization of Rademacher's theorem to Carnot groups.

\begin{theorem}[Pansu's Theorem \cite{Pan89}]\label{pansu}
Let $f\colon \bbG_{1}\to \bbG_{2}$ be a Lipschitz map between Carnot groups. Then $f$ is Pansu differentiable almost everywhere with respect to the Haar measure on $\bbG_{1}$.
\end{theorem}

In the case of Lipschitz curves (when $\bbG_{1}=\bbR)$, the group linear map $L$ has the form $h\mapsto \exp(hX)$ for some $X$ in the horizontal layer of $\bbG_{2}$ \cite[Lemma 2.1.4]{Mon01}. 
This motivates the following definition, adapted from \cite{JS16} to only consider directions in a subset of $\mathfrak{g}_1$.

\begin{definition}\label{directionalWhitney}
Suppose $\bbG$ is a Carnot group with horizontal layer $\mathfrak{g}_1$ and fix $\mathfrak{O} \subset \mathfrak{g}_1$. We will say that $\bbG$ has the {\em $C^1_H$ Whitney extension property on $\mathfrak{O}$} if the following implication is true:
for any compact $K \subset \mathbb{R}$, any continuous map $\gamma:K \to \bbG$, and any continuous map $X:K \to \mathfrak{O}$, 
if
\[ r_{K,\eta}:= \sup_{\substack{\tau, t\in K\\0<|\tau-t|<\eta}} \frac{ d(\gamma(t), \gamma(\tau)\exp((t-\tau)X(\tau)))}{|\tau-t|}\]
satisfies $r_{K,\eta}\to 0$ as $\eta \downarrow 0$,
then there is a curve $\Gamma \in C^1_H(\bbR,\bbG)$ such that $\Gamma|_{K}=\gamma$ and $\Gamma'|_{K}=X$.
\end{definition}
Note that this property acts as a converse to the uniform convergence of difference quotients of horizontal $C^1$ maps on compact sets (see \cite[Theorem 1.2]{Mag13} or \cite[Theorem~1.1]{Zim25}) just as Whitney's classical extension theorem acts as a converse to Taylor's theorem in the Euclidean setting.







\subsection{Engel Group}

Here, we describe a particular step 3 Carnot group that is not pliable. That is, it is a Carnot group in which there exist non-pliable horizontal vectors.
This group will be the focus of Section~\ref{sec-engel}.

\begin{definition}\label{Engeldef}
The Engel group $\mathbb{E}$ is the Carnot group whose Lie algebra $\mathfrak{e}$ has the decomposition
\[\mathfrak{e} =V_1 \oplus V_2 \oplus V_3,\]
where $V_1 = \text{span}\{X_1,X_2\}$, $V_2 = \text{span}\{X_3\}$, $V_3 = \text{span}\{X_4\}$, and the only non-zero Lie brackets among them are
\[[X_1,X_2] = X_3, \qquad [X_1,X_3] = X_4.\]
\end{definition}

It will be convenient to identify the Engel group with $\bbR^4$ using exponential coordinates of the second kind. In this setting, the horizontal vector fields $X_{1}$ and $X_{2}$ are given by
\begin{equation}\label{engelhorizontalvectors}
X_1(x)=\partial_1, \qquad X_2(x)=\partial_2 + x_1 \partial_3 +  \tfrac12 x_1^2\partial_4.
\end{equation}
The expressions for $X_{1}$ and $X_{2}$ seem to be well known, but we were not able to find an argument in the literature providing their derivation. Hence we give a brief justification in the appendix. Recall that we use $\dot{\gamma}$ to denote the Euclidean derivative of a curve in first or second exponential coordinates to distinguish it from $\gamma'$ which is identified with an element of the Lie algebra.

\begin{lemma}\label{Engellift}
An absolutely continuous curve $\gamma=(\gamma_{1},\gamma_{2},\gamma_{3},\gamma_{4}):[a,b] \to \mathbb{R}^{4}$ is horizontal in $\mathbb{E}$ if and only if
\[\dot{\gamma}_3(t) = \gamma_1(t) \dot{\gamma_2}(t) \quad \mbox{and} \quad \dot{\gamma}_4(t) = \tfrac12 (\gamma_1(t))^{2} \dot{\gamma}_2(t) \quad \mbox{for almost every }t\in [a,b].\]
\end{lemma}

\begin{proof}
Suppose $\gamma$ is horizontal in $\mathbb{E}$. Then for almost every $t\in [a,b]$ there are $a(t),b(t)\in \mathbb{R}$ such that $\dot{\gamma}(t)=a(t)X_{1}(\gamma(t))+b(t)X_{2}(\gamma(t))$. Using the representation of $X_{1}$ and $X_{2}$ above gives
\[\dot{\gamma}(t)=a(t)(1,0,0,0)+b(t) \left(0,1,\gamma_{1}(t),\tfrac{1}{2}(\gamma_{1}(t))^{2}\right).\]
Equating the first two components we see $a(t)=\dot{\gamma}_{1}(t)$ and $b(t)=\dot{\gamma}_{2}(t)$. Hence from the third and fourth components we see $\dot{\gamma}_{3}(t)=\gamma_{1}(t)\dot{\gamma}_{2}(t)$ and $\dot{\gamma}_{4}(t)=\frac{1}{2}(\gamma_{1}(t))^{2}\dot{\gamma}_{2}(t)$, as claimed. The converse is similar.
\end{proof}

\section{Whitney extension and Lusin approximation in a subset of directions}
\label{sec-whitney}


The following proposition is an analogue of Proposition 4.3 in \cite{JS16} and the proofs are similar. More details are included here for completeness (due to our variations in language and notation).

For convenience, if $\gamma:[a,b] \to \bbG$ is a differentiable curve, we will say that $\gamma$ is a curve {\em from $(x,X)$ to $(y,Y)$} if $\gamma(a) = x$, $\gamma'(a)=X$, $\gamma(b)=y$, and $\gamma'(b)=Y$. We denote the identity in $\bbG$ by $0$. 

If $Y \in \mathfrak{g}_1$ is pliable, then,
for any $\varepsilon > 0$, there is some $\eta(Y,\varepsilon) > 0$ such that, if
    $Z \in \mathfrak{g}_1$ and $z \in \bbG$ satisfy
    $$
    \|Z - Y\| < \eta(Y,\varepsilon) 
    \quad
    \text{and}
    \quad
    d(z,\exp(Y)) < \eta(Y,\varepsilon),
    $$
    then
    there exists a curve $\gamma \in C^1_H([0,1],\bbG)$ 
    from $(0,Y)$ to $(z,Z)$
    satisfying
    $\Vert \gamma' - Y \Vert_{\infty} < \varepsilon$.
    The following result tells us that, given a compact set of pliable directions, the parameter $\eta$ may be chosen to be the same for every vector in the compact set.

\begin{proposition}
\label{p-unifpliable}
    Suppose $\bbG$ is a Carnot group with horizontal layer $\mathfrak{g}_1$, and $\omega \subset \mathfrak{g}_1$ is a compact set such that every vector in $\omega$ is pliable.
    Then, 
    for any $\varepsilon > 0$,
    there exists $\eta > 0$ such that, for all $W \in \omega$, 
    if $Z \in \mathfrak{g}_1$ and $z \in \bbG$ satisfy 
    $$
    \|Z - W\| < \eta 
    \quad
    \text{and}
    \quad
    d(z,\exp(W)) < \eta,
    $$
    then
    there exists a curve $\gamma \in C^1_H([0,1],\bbG)$ from $(0,W)$ to $(z,Z)$ 
    satisfying
    $\Vert \gamma' - W \Vert_{\infty} < \varepsilon$.
\end{proposition}

\begin{proof}
    Fix $\varepsilon > 0$. We will prove the proposition using the following claim via the compactness of $\omega$:
    
    \medskip

    \noindent \textbf{Claim.}
    {\em
    For any $Y_0 \in \omega$, there is some 
    $\nu_{Y_0} > 0$
    such that, 
    if $W,Z \in \mathfrak{g}_1$ and $z \in \mathbb{G}$ satisfy
    $$
    \Vert W - Y_0 \Vert < \nu_{Y_0},
    \quad
    \Vert Z - Y_0 \Vert < \nu_{Y_0},
    \quad
    \text{and}
    \quad
    d(z,\exp(Y_0)) < \nu_{Y_0},
    $$
    then
    there exists a curve $\gamma \in C^1_H([0,1],\bbG)$ 
    from $(0,W)$ to $(z,Z)$
    satisfying
    $\Vert \gamma' - W \Vert_{\infty} < \varepsilon$.
    }

    \medskip

    \noindent \emph{Proof of Proposition from the Claim:} 
    For each $Y_0 \in \omega$, choose $\nu_{Y_0}$ to be the parameter guaranteed by the claim.
    Since $\exp:\mathfrak{g} \to \bbG$ is a diffeomorphism and the Euclidean and CC distances generate the same topology, we may find some 
    $r_{Y_0} \leq \nu_{Y_0}$ so that $d(\exp(W),\exp(Y_0)) < \tfrac12 \nu_{Y_0}$ for all $W \in B_{\mathfrak{g}_1}(Y_0,r_{Y_0})$. 
    Consider the cover $\{B_{\mathfrak{g}_1}(Y_0,\tfrac12 r_{Y_0}) : Y_0 \in \omega\}$ of $\omega$, and choose a finite subcover $\{B_{\mathfrak{g}_1}(Y_i,\tfrac12r_i)\}_{i=1}^I$, where $r_{i}:=r_{Y_{i}}$.
    Set $\eta =  \min \{\tfrac12 r_i: i=1,\dots, I \}$. 
    
    Fix any $W \in \omega$, $Z \in B_{\mathfrak{g}_1}(W,\eta)$, and $z \in B_{\bbG}(\exp(W),\eta)$. Using the definition of a subcovering, we can choose $k \in \{1,\dots,I\}$ such that $\Vert W - Y_k \Vert < \tfrac12 r_k < \nu_{Y_k}$. Since $\eta\leq \tfrac12 r_{k} \leq \tfrac12 \nu_{Y_k}$, it follows that $\Vert Z - Y_k \Vert < \nu_{Y_k}$. 
    Moreover, 
    $$
    d(z,\exp(Y_k)) \leq d(z,\exp(W)) + d(\exp(W),\exp(Y_k)) < \eta + \tfrac12 \nu_{Y_k} \leq \nu_{Y_k}
    $$
    since $\|W-Y_k\|< r_{k}$. 
    According to the claim, then, there exists a curve $\gamma \in C^1_H([0,1],\bbG)$ 
    from $(0,W)$ to $(z,Z)$
    satisfying
    $\Vert \gamma' - W \Vert_{\infty} < \varepsilon$ which proves the proposition.


    \medskip

    \noindent
\emph{Proof of the Claim:} Fix $Y_0 \in \omega$.
Using the pliability of $Y_0$, choose $\eta(Y_0,\varepsilon) > 0$ such that, if
    $Z \in \mathfrak{g}_1$ and $z \in \bbG$ satisfy
    $$
    \|Z - Y_0\| < \eta(Y_0,\varepsilon) 
    \quad
    \text{and}
    \quad
    d(z,\exp(Y_0)) < \eta(Y_0,\varepsilon),
    $$
    then
    there exists a curve $\gamma \in C^1_H([0,1],\bbG)$ 
    from $(0,Y_0)$ to $(z,Z)$
    satisfying
    $\Vert \gamma' - Y_0 \Vert_{\infty} < \tfrac{\varepsilon}{4}$.

    As above, we may 
    choose $r_{Y_0,\varepsilon} >0$ so that, for all $Y \in \mathfrak{g}_1$ with $\|Y-Y_0\| < r_{Y_0,\varepsilon}$, we have
    $$
    d(\exp(Y),\exp(Y_0)) < \tfrac18 \eta(Y_0,\varepsilon).
    $$
    
    Set $\nu := \nu_{Y_0} := \min\{ \tfrac{\varepsilon}{4}, \tfrac18 \eta(Y_0,\varepsilon), r_{Y_0,\varepsilon}\}$, and fix $W,Z \in B_{\mathfrak{g}_1}(Y_0,\nu)$ and $z \in B_{\mathbb{G}}(\exp(Y_0),\nu)$.
    We must construct a curve as in the claim.
    For any $a \in \mathbb{G}$ with $d(a,z) < \tfrac12 \eta(Y_0,\varepsilon)$,
    our hypothesis guarantees that
    there is
    a curve $\gamma_a \in C^1_H([0,1],\bbG)$ 
    from $(0,Y_0)$ to $(a,Z)$
    satisfying
    $\Vert \gamma_a' - Y_0 \Vert_{\infty} < \tfrac{\varepsilon}{4}$.

    For any $0< \rho < 1$, 
    we can
    define a curve $\gamma_\rho:[0,\rho] \to \mathbb{G}$ so that
    $\gamma_\rho(0)=0$ and
    $$
    \gamma'_{\rho}(t) = 
    \tfrac{1}{\rho} (tY_0 + (\rho-t)W)\quad \mbox{for all }t\in [0,\rho].
    $$
    Notice that this curve does not depend on the choice of $a$. Write $x_\rho := \gamma_{\rho}(\rho)$.
    Define now the curve 
    $\gamma_{a,\rho}:[0,1]\to\mathbb{G}$ as
    $$
    \gamma_{a,\rho}(t) = 
    \begin{cases}
    \gamma_\rho(t) & \text{ for } t \in [0,\rho]\\
    x_\rho \delta_{1-\rho}(\gamma_a(\tfrac{1}{1-\rho} (t - \rho))) & \text{ for } t \in [\rho,1]
    \end{cases}
    $$
    so that 
    $\gamma_{a,\rho}'(t) = \gamma_a'(\tfrac{1}{1-\rho} (t - \rho))$ for $t \in [\rho,1]$
    (see (T1) on page 1646 of \cite{JS16}).
    In particular, $\gamma_{a,\rho}'(\rho)$ exists and is equal to $Y_0$.
    It follows that $\gamma_{a,\rho} \in C^1_H([0,1],\mathbb{G})$ 
    is a curve from $(0,W)$ to $(x_\rho \delta_{1-\rho}(a),Z)$.
    Moreover, if $t\in[0,\rho]$ we have
    \[     \Vert \gamma_{a,\rho}'(t) - W \Vert = \Vert \tfrac{1}{\rho}(tY_{0}-tW) \Vert \leq \|W-Y_{0}\| < \tfrac{\varepsilon}{4}. \]
If $t\in [\rho,1]$ we have
    $$
    \Vert \gamma_{a,\rho}'(t) - W \Vert
    \leq 
    \Vert \gamma_a' - Y_0 \Vert_\infty + \|W - Y_0\|
    < \tfrac{\varepsilon}{2},
    $$
Combining these gives $\Vert \gamma_{a,\rho}' - W \Vert_\infty < \varepsilon$.
    
    To ensure that the curve $\gamma_{a,\rho}$ is the curve required in the claim, we must be able to choose $a$ in such a way that $x_\rho  \delta_{1-\rho}(a) = \gamma_{a,\rho}(1) = z$.
    In other words, if we can choose some $\rho$ so that the point $a_\rho := \delta_{(1-\rho)^{-1}} (x_\rho^{-1}  z)$ satisfies 
    $d(a_\rho,z) < \tfrac12 \eta(Y_0,\varepsilon)$,
    then we may conclude
    $\gamma_{a_\rho,\rho}(1) = x_\rho  \delta_{1-\rho}(\delta_{(1-\rho)^{-1}} (x_\rho^{-1}  z)) = z$, and hence $\gamma_{a_\rho,\rho}$ is the desired curve.
    
    It remains to prove that such a $\rho$ exists i.e. that $\rho>0$ may be chosen with
    \[d(\delta_{(1-\rho)^{-1}} (x_\rho^{-1}  z),z) < \tfrac12 \eta(Y_0,\varepsilon).\]
    To see that this is possible, notice that $d(0,x_\rho) \leq \rho \max\{ \|W\|,\|Y_0\|\}$, which implies $x_{\rho}\to 0$ as $\rho \to 0$. Hence
\begin{align*} 
d( \delta_{(1-\rho)^{-1}} (x_\rho^{-1}  z) , z)&= \tfrac{1}{1-\rho} d( x_{\rho}^{-1}z,\delta_{1-\rho}(z))\\
&\leq \tfrac{1}{1-\rho}\left( d(x_{\rho}^{-1}z,z)+d(z,\delta_{1-\rho}(z)) \right)
\longrightarrow 0
\quad
\text{ as } \rho \to 0
\end{align*}
since $x_{\rho}\to 0$ and $\delta_{1-\rho}(z)\to z$ as dilations are continuous with respect to the scaling parameter.
\end{proof}

Before proving the main ``partial Whitney'' result, we will verify that ``Carnot Whitney'' implies ``classical Whitney'' at 0 when viewed in exponential coordinates of the first kind. Note that that, in these coordinates, the identity is $0$ and $g^{-1}=-g$ for any $g\in \bbG$. Recall that $|\cdot|$ denotes the Euclidean norm on $\mathbb{R}^n$.

    \begin{lemma}
    \label{l-toostrong}
        Suppose $K \subset \mathbb{R}$ is compact and $f:K\to \bbG$ and $X:K \to \mathfrak{g}_1$ are continuous. Then there is a constant $C>0$ 
        such that the following implication holds.
        
        Suppose $\varepsilon>0$ and $x,y \in K$ satisfy $0<y-x < 1$, $f(x)=0$, and
        \begin{equation}
        \label{e-hypothesis2}
                    \frac{d(f(y),\exp((y-x)X(x)))}{y-x} < \varepsilon.
        \end{equation}
Then
        $$
    \frac{|f(y) - \exp((y-x)X(x))|}{y-x}
    < C\varepsilon.
    $$
    Note that $C$ is independent of $x$, $y$, and $\varepsilon$. Here, as throughout this section, we view $\bbG$ in exponential coordinates of the first kind.
    \end{lemma}
    
    \begin{proof}
    Throughout this proof, we will simplify notation by writing $X := X(x)$.
Using the hypothesis and the fact that $d$ is bi-Lipschitz equivalent to the box metric $d_{Box}$ on $\bbG$ (in the sense of \eqref{e-supmetric}), there is a constant $C>0$ independent of $x$, $y$, and $\varepsilon$ such that
\begin{align}
 \Vert f(y)^{-1}\exp((y-x)X) \Vert_{Box} 
 =
 d_{Box}(\exp((y-x)X),f(y))
 &\leq  C d(\exp((y-x)X),f(y)) \nonumber\\
 &< C\varepsilon (y-x). \label{e-box}
\end{align}
Using the fact that the group operation is equal to Euclidean addition in the first $r$ coordinates (Proposition~\ref{p-structure}) and the definition of $\Vert \cdot \Vert_{Box}$, we have for some possibly larger constant
    $$
    \left| \frac{f_k(y) - \exp((y-x)X)_k}{y -x} \right| < C\varepsilon
    $$
    for $k \leq r$.
    Suppose $k>r$ and that we have proven for all $j <k$ that 
    $$
    \left| \frac{f_j(y) - \exp((y-x)X)_j}{y -x} \right| < C \varepsilon
    $$
    for some constant $C$ depending only on $f$, $X$, $K$, and the coordinate representation of $\mathbb{G}$. We show how to obtain the same estimate for $j=k$ and hence prove the statement by induction. Since $X \in \mathfrak{g}_1$, we have $\exp((y-x)X)_k = 0$. Thus we need only to bound 
    $\left| \frac{f_k(y)}{y -x} \right|$
    in terms of $\varepsilon$.
    Using Proposition~\ref{p-structure}, we have
    \begin{align}
    \left| \frac{f_k(y)}{(y -x)} \right| 
    &\leq
    \left| \frac{f_k(y) - Q_k(f(y)^{-1},\exp((y-x)X))}{(y -x)^{d_k}} \right| (y-x)^{d_k-1} \label{e-easybound}\\
    &\qquad +
    \left| \frac{Q_k(f(y)^{-1},\exp((y-x)X))}{y-x} \right|\label{e-alsoeasybound}
        \end{align}
      where $d_k$ is the homogeneity of coordinate $k$. 
    According to \eqref{e-box}, the term on the right of \eqref{e-easybound} is bounded by $C \varepsilon^{d_k} (y-x)^{d_k-1} < C\varepsilon$.
    To bound \eqref{e-alsoeasybound},
    recall from Proposition~\ref{p-structure} that $Q_k(f(y)^{-1},\exp((y-x)X))$ is a sum of polynomials each of which contains a factor of the form
    \begin{align*}
    \exp((y-x)X)_i f_\ell(y)
    &-
    \exp((y-x)X)_\ell f_i(y)\\
    &=
    [\exp((y-x)X)_i - f_i(y)] f_\ell (y)
    -
    [\exp((y-x)X)_\ell - f_\ell (y)] f_i(y)
    \end{align*}
    for some $i,\ell < k$.
    According to the induction hypothesis, then, we may bound \eqref{e-alsoeasybound} by $C\varepsilon$ for a possibly larger $C$ depending on the same data.
    \end{proof}

   From here, we can conclude that ``Carnot Whitney'' implies ``classical Whitney'' everywhere in a compact set $K$. Below, we will refer to a ``modulus of continuity'' $\alpha$. By this, we simply mean a non-decreasing function $\alpha~\colon~[0,\infty)~\to~[0,\infty)$ that is continuous at $0$ and $\alpha(0)=0$.
   We say that $f:\mathbb{R}^m \supset A \to \mathbb{R}^k$ is uniformly continuous with modulus of continuity $\alpha$ if $|f(b)-f(a)| \leq \alpha(|b-a|)$ for all $a,b \in A$. 

    \begin{lemma}
    \label{l-toostrong2}
        Suppose $K \subset \mathbb{R}$ is compact and $f:K\to \bbG$ and $X:K \to \mathfrak{g}_1$ are continuous. Then there is a modulus of continuity $\alpha$ 
        such that the following implication holds. 
        
        Suppose $0<\varepsilon<1$ and
         $x,y \in K$ satisfy $0<y-x < \varepsilon$ and
        \begin{equation}
        \label{e-hypothesis}
                    \frac{d(f(y),f(x)\exp((y-x)X(x)))}{y-x} < \varepsilon.
        \end{equation}
        Then
        $$
    \frac{|f(y) - f(x) - (y-x)X(x)(f(x))|}{y-x}
    < \alpha(\varepsilon).
    $$
        Note that $\alpha$ is independent of $x$, $y$, and $\varepsilon$. Here, as throughout this section, we view $\bbG$ in exponential coordinates of the first kind.
    \end{lemma}
\begin{proof}
        Write $g(t) = f(x)^{-1} f(t)$ for all $t \in K$. 
        In particular, this gives $g(x) = 0$ and
        \begin{equation}
         \label{e-gdq}
        \frac{d(g(y),\exp((y-x)X(x)))}{y-x} 
        =
        \frac{d(f(y),f(x)\exp((y-x)X(x)))}{y-x}
        < \varepsilon.
                \end{equation}
The left invariance of $X(x) \in \mathfrak{g}_1$ and the fact that $X(x)(0)$ is identified with $\exp(X(x))$ in exponential coordinates of the first kind implies that
    $$
    X(x)(f(x)) 
    =
    D_0 L_{f(x)}(\exp(X(x)))
    $$
    where $L_p$ is the operation of left multiplication by $p \in \bbG$.  If $f(y) = f(x)$, then the conclusion is immediate from this point. 
    Assume $f(y) \neq f(x)$.
    Then, 
    (writing $L$ in place of $L_{f(x)}$, $D$ in place of $D_{0}$, and $X$ in place of $X(x)$) we have
    \begin{align}
        &\frac{|f(y) - f(x) - (y-x)X(f(x))|}{y-x} \nonumber \\
        & \qquad 
        =
        \frac{|L( g(y)) - L( g(x)) - 
        DL(\exp((y-x)X))
        |}{y-x} \nonumber \\
        & \qquad
        \leq
        \frac{|L( g(y)) - L( g(x)) - 
        DL(g(y)-g(x))
        |}{|g(y)-g(x)|}\cdot
        \frac{|g(y)-g(x)|}{|y-x|} \label{e-first2} \\
        & \qquad \qquad \qquad \qquad \qquad \qquad
        +
        \left|DL\left(\frac{g(y) - 
        \exp((y-x)X)
        }{y-x}\right)\right| \nonumber
    \end{align}

    Notice that $R:\mathbb{R}^n \times \mathbb{R}^n \to \mathbb{R}^n$ defined as $R(a,b) = L_a(b)$ is a polynomial. Hence the norm of the matrix $D_0 L_{z}$ is uniformly bounded for all $z$ in the compact set $f(K)$.
    It follows from this and from \eqref{e-gdq} and Lemma~\ref{l-toostrong} that
    the last term above is bounded by a constant multiple of $\varepsilon$ where the constant depends only on $f$, $X$, $K$, and the coordinate representation of $\mathbb{G}$.  
    
    We will now verify that \eqref{e-first2} is controlled as well.
    The second term \eqref{e-first2} is bounded by a constant independent of $x$ and $y$ due to Lemma~\ref{l-toostrong}.
    Notice also that the difference quotients of the polynomial $R$ converge uniformly on the compact set $f(K) \times g(K)$. 
    Indeed, if $B \subset \mathbb{R}^{2n}$ is any ball containing $f(K) \times g(K)$, then, for each $k \in \{ 1,\dots,n\}$, the gradient $\nabla R_k$ is uniformly continuous on $B$ with a modulus of continuity $\beta$ depending on $B$ (which can be made to depend only on $f$, $K$, and the coordinate representation of $\mathbb{G}$). 
    Hence, for any $a \in f(K)$ and $b,c \in g(K)$, we have 
    \begin{align*}
        |R_k(a,b) - R_k(a,c) &- \nabla R_k (a,c) \cdot ((a,b)-(a,c))|\\
        &=
         \left| \int_0^1 \nabla R_k (\gamma(t)) \cdot \gamma'(t)  - \nabla R_k (a,c) \cdot ((a,b)-(a,c)) \, dt\right|\\
         &\leq
         |(a,b)-(a,c)| \int_0^1 |\nabla R_k(\gamma(t)) - \nabla R_k(a,c)| \, dt\\
         & \leq
        |b-c| \beta(|b-c|),
    \end{align*}
    where $\gamma:[0,1] \to \mathbb{R}^{2n}$ is the segment from $(a,c)$ to $(a,b)$.
    Setting $a=f(x)$, $b=g(y)$, and $c=g(x)$, the above estimate shows that the $k$'th component of $L( g(y)) - L( g(x)) - 
        DL(g(y)-g(x))$ has absolute value bounded by $|g(y)-g(x)|\beta(|g(y)-g(x)|)$. 
        We then apply the uniform continuity of $g$ on $K$ to conclude that 
    there is a modulus of continuity $\hat{\alpha}$ depending only on $f$, $X$, $K$, and the coordinate representation of $\mathbb{G}$ such that
    the first term in \eqref{e-first2} is bounded by $\hat{\alpha}(\varepsilon)$. 
    The desired modulus of continuity is then obtained by scaling $\hat{\alpha}$ by a constant. 
\end{proof}

We are now ready to prove Theorem~\ref{t-partialwhitney}. We restate it here for convenience.

\begin{theorem*}[Restatement of Theorem~\ref{t-partialwhitney}]
    Suppose $\bbG$ is a Carnot group with horizontal layer $\mathfrak{g}_1$, and fix $\mathfrak{O} \subset \mathfrak{g}_1$.
    If every vector in $\mathfrak{O}$ is pliable, then $\bbG$ has the $C^1_H$ Whitney extension property on $\mathfrak{O}$.
\end{theorem*}

\begin{proof}
Suppose $K \subset \mathbb{R}$ is compact and that $\gamma:K \to \bbG$ and $X:K \to \mathfrak{O}$ are continuous maps such that
\[ r_{K,\eta}:= \sup_{\substack{\tau, t\in K\\0<|\tau-t|<\eta}} \frac{ d(\gamma(t), \gamma(\tau)\exp((t-\tau)X(\tau)))}{|\tau-t|}\]
satisfies $r_{K,\eta}\to 0$ as $\eta \downarrow 0$.

    The set $\omega = X(K)$ is compact in $\mathfrak{O}$.
    Thus,
    for any $n \in  \mathbb{N}$, we may use Proposition~\ref{p-unifpliable} to choose some $0< \eta_n < \tfrac{1}{n}$ such that, for all $W \in \omega$ and all $Z \in B_{\mathfrak{g}_1}(W,\eta_n)$ and $z \in B_{\bbG}(\exp(W),\eta_n)$, there exists a curve $\gamma \in C^1_H([0,1],\bbG)$ from $(0,W)$ to $(z,Z)$ 
    such that $\Vert \gamma' - W \Vert_{\infty} < \frac{1}{n}$.
    We may then apply the uniform continuity of $X$ on $K$ to choose some $0<\alpha_n < \tfrac{1}{n}$ so that, whenever $x,y \in K$ satisfy $|x-y|<\alpha_n$, we have
    $
    \|X(x) - X(y)\| < \eta_n
    $
    and $r_{K,\alpha_n} < \eta_n$. 


    Write $[\min K , \max K] \setminus K$ as the union of disjoint open intervals $(a_i,b_i)$ for $i \in \mathbb{N}$.     
    Fix some $i \in \mathbb{N}$ and write $(a,b) := (a_i,b_i)$. 
    Suppose for the moment that $b-a < \alpha_1$.
    Choose the largest integer $n \in \mathbb{N}$ so that $|b-a| < \alpha_n$. Write $z= \delta_{1/(b-a)}(\gamma(a)^{-1} \gamma(b))$.
    Then
    \begin{equation}
        \label{e-vertvexctor}
    \|X(b)-X(a)\| < \eta_n
    \end{equation}
    and 
    \begin{align*}
        d(z,\exp(X(a))) = \frac{d(\gamma(b),\gamma(a) \exp{((b-a)X(a))}}{b-a} \leq r_{K,\alpha_n}
        < \eta_n.
    \end{align*}
    It follows that there is some curve $\gamma \in C^1_H([0,1],\bbG)$ from $(0,X(a))$ to $(z,X(b))$ such that
    $\Vert \gamma' - X(a) \Vert_{\infty} < \frac{1}{n}$.
    Then, for all $t \in [a,b]$,
    define $\Gamma(t) := \gamma(a)\delta_{(b-a)}\left(\gamma(\frac{t-a}{b-a})\right)$.
    Hence $\Gamma|_{[a,b]}$ is a curve from $(\gamma(a),X(a))$ to $(\gamma(b),X(b))$, 
    and
    \begin{equation}
        \label{e-boundonderiv}
        \Vert \Gamma' - X(a) \Vert_{\infty} < \tfrac{1}{n}
    \end{equation}
    by (T1) on page 1646 of \cite{JS16}. 
    
    If $b-a \geq \alpha_1$, choose $\gamma$ to be any curve in $C^1_H([0,1],\mathbb{G})$ from $(0,X(a))$ to $(z,X(b))$ and define $\Gamma$ as above. Since $K$ is bounded, there are only finitely many such intervals and they will not play an important role when considering convergences at small scales.
    
    Applying this procedure to $(a_{i},b_{i})$ for each $i \in \mathbb{N}$ defines a map $\Gamma:[\min K , \max K] \setminus K \to \mathbb{G}$. We extend the definition of $\Gamma$ to all of $[\min K , \max K]$ by defining $\Gamma|_K := \gamma|_K$.

    We know by construction that $\Gamma$ is $C^1$ and horizontal on $(a_i,b_i)$ for each $i\geq 1$. It remains to prove that $\Gamma$ is differentiable at any $x \in K$ with $\Gamma'(x) = X(x)$ and that $\Gamma'$ is continuous on $[\min K, \max K]$. For the remainder of the proof, consider $\Gamma$ as a curve in the ambient space $\mathbb{R}^n$ using exponential coordinates of the first kind. In particular, we consider $X(\cdot)$ to be a left invariant vector field on $\mathbb{R}^n$. 

    Fix a point $x\in K$.
    We will first show that the Euclidean derivative $\dot{\Gamma}$ of $\Gamma$ at $x$ exists and is equal to $X(x)(\Gamma(x))$. To simplify (while abusing) notation, we will still write $X(t)$ to represent $X(t)(\Gamma(t))$ for any $t \in K$. Suppose $\{x_\ell\}_{\ell=1}^{\infty}$ is a sequence in $\mathbb{R}$ converging to $x$. We must prove that 
    \begin{equation}
    \label{e-converges}        
    \frac{| \Gamma(x_\ell) - \Gamma(x) - (x_\ell-x) X(x)|}{x_\ell - x} \to 0
    \end{equation}
    as $\ell \to \infty$.
    Assume that $\{x_\ell\}$ is decreasing. (The proof in the case when $\{ x_\ell \}$ is increasing is similar.) 
    If all $x_\ell$ lie in one interval $(a_k,b_k)$, then $x = a_k$, so \eqref{e-converges} follows from the fact that $\Gamma$ is $C^1$ on $(a_k,b_k)$ and $\Gamma$ is differentiable from the right at $a_{k}$ with derivative equal to $X(a_{k})(\Gamma(a_{k}))$.
    If all $x_\ell$ lie in $K$, then \eqref{e-converges} follows from Lemma~\ref{l-toostrong2}.
    It suffices (by passing to subsequences if necessary) to consider the case when each $x_\ell$ lies in a different interval $(a_{k_\ell},b_{k_\ell})$ for each $\ell$. 
    We then have 
    \begin{align}
        | \Gamma(x_\ell) - \Gamma(x) - (x_\ell-x) X(x)| &\leq 
        | \Gamma(x_\ell) - \Gamma(a_{k_\ell}) - (x_\ell-a_{k_\ell}) X(a_{k_\ell}) | \label{e-smooth1}\\
        & \qquad + | \gamma(a_{k_\ell}) - \gamma(x) - (a_{k_\ell}-x)X(x) | \label{e-smooth2}\\
        & \qquad + (x_\ell-a_{k_\ell})| X(a_{k_\ell}) - X(x)|. \label{e-smooth3}
    \end{align}
    Since $\Gamma$ is $C^1$ on $(a_{k_\ell},b_{k_\ell})$ and since $\Gamma$ is differentiable from the right at $a_{k_\ell}$ with derivative equal to $X(a_{k_\ell})(\Gamma(a_{k_\ell}))$, we may bound the term on the right-hand side of \eqref{e-smooth1} by
    $$
    \int_{a_{k_\ell}}^{x_\ell} |\dot{\Gamma}(t) - \dot{\Gamma}(a_{k_\ell})| \, dt
    \leq
    \sup_{t \in [a_{k_\ell},b_{k_\ell}]} |\dot{\Gamma}(t) - \dot{\Gamma}(a_{k_l})|(x_\ell - a_{k_\ell}).
    $$Since $x_\ell - a_{k_\ell} \leq x_\ell -x$, this quantity is $o(x_\ell-x)$ due to \eqref{e-boundonderiv}.
    The term in \eqref{e-smooth3} is $o(x_\ell-x)$ since $X$ is continuous on $K$.
    The fact that \eqref{e-smooth2} is $o(x_\ell-x)$ follows from Lemma~\ref{l-toostrong2}.
    Therefore, $\dot{\Gamma}(x)$ exists and is equal to $X(x)(\Gamma(x))$ for all $x \in K$.
    
    Finally, we will show that $\dot{\Gamma}$ is continuous on $[\min K, \max K]$ i.e.
        \begin{equation}
        \label{e-ctsderiv}
        |\dot{\Gamma}(x_\ell) - \dot{\Gamma}(x)| \to 0
    \end{equation}
    for any $x\in K$ and any sequence $x_\ell \to x$. Suppose $x\in K$, since otherwise the statement is obvious.
    As above, we may assume that $\{x_\ell\}$ is decreasing.
    If all $x_\ell$ lie in one interval $(a_k,b_k)$, then \eqref{e-ctsderiv} follow from the fact that $\Gamma$ is $C^1$ on $(a_k,b_k)$ and $\dot{\Gamma}(a_k) = X(a_k)(\Gamma(a_k))$.
    If all $x_\ell$ lie in $K$, then \eqref{e-ctsderiv} follows from our assumption that $X$ is continuous on $K$ and $\Gamma' = X$ on $K$.
    In the case when each $x_\ell$ lies in a different interval $(a_{k_\ell},b_{k_\ell})$ for each $\ell$,
    we have    
    \begin{align*}
    |\dot{\Gamma}(x_\ell) - \dot{\Gamma}(x)|
    &\leq
    |\dot{\Gamma}(x_\ell) - X(a_{k_\ell})(\Gamma(a_{k_\ell}))|
    +
    |X(a_{k_\ell})(\Gamma(a_{k_\ell})) - X(x)(\Gamma(x)|.
    \end{align*}
    The first term on the right vanishes as $\ell \to \infty$ 
    due to \eqref{e-boundonderiv}, and the second term vanishes
    by our assumption that $X$ is continuous on $K$.

    In summary, we have shown that $\dot{\Gamma}(x)$ exists, is equal to $X(x)$, and is continuous for all $x \in K$. It follows from the construction of $\Gamma$ and definition of $X$ that $\Gamma$ is $C^1$ on $[\min K, \max K]$ and $\Gamma'(t)\in \mathfrak{g}_{1}$ for every $t\in [\min K, \max K]$. Also $\gamma$ extends $\Gamma$ and $\gamma'$ extends $X$ to $[\min K, \max K]$. Finally we extend $\Gamma$ arbitrarily off $[\min K, \max K]$ to a $C_{H}^{1}$ curve on $\mathbb{R}$.
\end{proof}

\section{Pliability and Lusin approximation in the Engel Group}
\label{sec-engel}

Recall the Engel group has Lie algebra $V_{1}\oplus V_{2} \oplus V_3$, where the horizontal layer $V_{1}$ is spanned by $X_{1}$ and $X_{2}$ (Definition \ref{Engeldef}). We identify it with $\mathbb{R}^{4}$ in exponential coordinates of the second kind. We first classify those directions that are pliable.

\begin{theorem}
\label{t-engel}
Given $a,b\in \mathbb{R}$, let $V=aX_{1}+bX_{2}$ be an element of the horizontal layer $V_{1}$ of the Engel group. Then $V$ is pliable if and only if either $a\neq 0$ or $a=b=0$. 
\end{theorem}

\begin{proof}
\emph{Suppose  $a=b=0$.} In this case, pliability follows because the zero vector is pliable in any Carnot group \cite[Proposition 6.1]{JS16}. 

\noindent \textbf{Case 1.} Suppose $a,b\in \mathbb{R}$ with $a\neq 0$. 

 Given $c, d_{1}, d_{2}, d_{3} \in \mathbb{R}$, consider $\gamma=\gamma_{c,d_1,d_2,d_3} \colon [0,1]\to \mathbb{R}^{4}$ defined by $\gamma(0)=0$ and, for all $t\in [0,1]$,
\[ \dot{\gamma}_{1}(t)=a+ct, \qquad \dot{\gamma}_{2}(t)=b+d_{1}t+d_{2}t^{2}+d_{3}t^{3}\]
and
\[ \dot{\gamma}_{3}(t)=\dot{\gamma}_{2}(t)\gamma_{1}(t), \qquad \dot{\gamma}_{4}(t)=\frac{1}{2}\dot{\gamma}_{2}(t)(\gamma_{1}(t))^{2}.\]
By Lemma \ref{Engellift}, $\gamma$ is horizontal. Clearly $\gamma \in C_{H}^{1}([0,1],\mathbb{E})$ and \[\dot{\gamma}(0)=(a,b,0,0)=aX_{1}(0)+bX_{2}(0)=V(0).\]

We next compute $\gamma(1)$. Later in the proof, we will apply the inverse function theorem during which terms that are polynomial in $c, d_{1}, d_{2}, d_{3}$ of degree two or higher will be irrelevant. For that reason, we denote such terms by $R$. Note that the exact formula for $R$ may be different in each line.

To begin, clearly by integration 
\[\gamma_{1}(1)=a+\frac{c}{2} \quad \mbox{and} \quad \gamma_{2}(1)=b+\frac{d_{1}}{2}+\frac{d_{2}}{3}+\frac{d_{3}}{4}.\]
Since also $\gamma_{1}(t)=at+\frac{c}{2}t^2$, we have
\begin{align*}
\gamma_{3}(1)=\int_{0}^{1} \dot{\gamma}_{2}\gamma_{1}&= \int_{0}^{1} (b+d_{1}t+d_{2}t^{2}+d_{3}t^{3})(at+\frac{c}{2}t^2) dt \\
&= \int_{0}^{1} ab t + \frac{bc}{2}t^{2}+ad_{1}t^{2}+ad_{2}t^{3}+ad_{3}t^{4} dt + R\\
&=\frac{ab}{2}+\frac{bc}{6}+\frac{ad_{1}}{3}+\frac{ad_{2}}{4}+\frac{ad_{3}}{5}+R.
\end{align*}
Similarly,
\begin{align*}
\gamma_{4}(1)=\frac{1}{2}\int_{0}^{1} \dot{\gamma}_{2}(\gamma_{1})^{2}&=\frac{1}{2} \int_{0}^{1} (b+d_{1}t+d_{2}t^{2}+d_{3}t^{3})(at+\frac{c}{2}t^2)^{2} dt \\
&=\frac{1}{2}\int_{0}^{1}(b+d_{1}t+d_{2}t^{2}+d_{3}t^{3})(a^{2}t^{2}+act^{3}) dt + R\\
&=\frac{1}{2}\int_{0}^{1} a^{2}bt^{2}+abct^{3}+a^{2}d_{1}t^{3}+a^{2}d_{2}t^{4}+a^{2}d_{3}t^{5} dt + R\\
&=\frac{a^{2}b}{6}+\frac{abc}{8}+\frac{a^{2}d_{1}}{8}+\frac{a^{2}d_{2}}{10}+\frac{a^{2}d_{3}}{12}+R.
\end{align*}
Now consider the map $F\colon \mathbb{R}^{4}\to \mathbb{R}^{4}$ defined by $F(c,d_{1},d_{2},d_{3})=(\gamma_{1}(1),\gamma_{2}(1),\gamma_{3}(1),\gamma_{4}(1))$. Notice $F(0,0,0,0)=(a,b,\frac{ab}{2},\frac{a^2 b}{6})=\exp(V)$.
The derivative of $F$ at $(0,0,0,0)$ is given by
\[dF(0,0,0,0)=\begin{pmatrix}
\frac{1}{2} & 0 & 0 & 0\\
0 & \frac{1}{2} & \frac{1}{3} & \frac{1}{4}\\
\frac{b}{6} & \frac{a}{3} & \frac{a}{4} & \frac{a}{5}\\
\frac{ab}{8} & \frac{a^2}{8} & \frac{a^2}{10} & \frac{a^2}{12}
\end{pmatrix}. \]
As noted before, the terms contained in the polynomials denoted by $R$ do not contribute to the expression for $dF(0,0,0,0)$. The determinant of $dF(0,0,0,0)$ is equal to
\[ \frac{1}{2}a^{3} \begin{vmatrix}
\frac{1}{2} & \frac{1}{3} & \frac{1}{4}\\
\frac{1}{3} & \frac{1}{4} & \frac{1}{5}\\
\frac{1}{8} & \frac{1}{10} & \frac{1}{12}
\end{vmatrix} = \frac{1}{2}a^{3} \cdot \frac{1}{86400}\neq 0 \]
since we assumed $a \neq 0$.
Hence, by the inverse function theorem, there exists an open set $U_1 \subset \mathbb{R}^{4}$ containing $(0,0,0,0)$ and an open set $U_2 \subset \mathbb{R}^{4}$ containing $F(0,0,0,0)=\exp(V)$ such that $F|_{U_1}$ is a $C^{1}$ bijection from $U_1$ onto $U_2$ with $C^{1}$ inverse. It follows that whenever $r>0$ satisfies $B(0,r)\subset U_1$, then $F(B(0,r))$ is an open set containing $F(0)$.

To show $V$ is pliable it suffices to show that for every neighborhood $\mathcal{V}$ of the curve $t\mapsto L(t):= \exp(tV)$ in $C_{H}^{1}([0,1],\mathbb{E})$, the set
\[\{\beta(1):\beta \in \mathcal{V}, (\beta, \beta')(0)=(0,V)\}\]
is a neighborhood of $\exp(V)$ \cite[Proposition 3.7]{JS16}. Fix such a neighborhood $\mathcal{V}$. 

\medskip

\noindent \textbf{Claim:} If $r>0$ is sufficiently small then $\gamma = \gamma_{c,d_1,d_2,d_3}\in \mathcal{V}$ for all $(c,d_1,d_2,d_3) \in B(0,r)$.

\medskip

We will first show how the claim can be used to prove the theorem.
Suppose we can choose $r>0$ small enough such that $\gamma = \gamma_{c,d_1,d_2,d_3}\in \mathcal{V}$ for all $(c,d_1,d_2,d_3) \in B(0,r)$ and $B(0,r)\subset U_1$. Then 
\[\{\beta(1):\beta \in \mathcal{V}, (\beta, \beta')(0)=(0,V)\}\] contains \[F(B(0,r))=\{\gamma_{c,d_1,d_2,d_3}(1):(c,d_1,d_2,d_3) \in B(0,r)\}.\] This is an open set containing $\exp(V)$, so it follows that $V$ is pliable.

\medskip

\noindent \emph{Proof of Claim:}
Notice $L'(t)=V=aX_{1}+bX_{2}$ and $\gamma'(t)=(a+ct)X_{1}+(b+d_{1}t+d_{2}t^{2}+d_{3}t^{3})X_{2}$. Hence
\[\|\gamma'(t)-L'(t)\|=\sqrt{c^{2}t^{2}+(d_{1}t+d_{2}t^{2}+d_{3}t^{3})^{2}}.\]
After expanding the square, we see given any $\varepsilon>0$ there is $\delta>0$ such that $0<r<\delta$ implies $\sup_{t\in [0,1]} \| \gamma'(t)-L'(t)\|\leq \varepsilon$ for all $(c,d_1,d_2,d_3) \in B(0,r)$.

It remains to show that, for any $\varepsilon>0$, there is some $\delta>0$ such that $0<r<\delta$ implies \[\sup_{t\in [0,1]}d(\gamma(t),L(t))\leq \varepsilon \qquad \mbox{for all }(c,d_1,d_2,d_3) \in B(0,r).\] Recall that the Euclidean distance $d_{e}$ and the CC distance $d_{cc}$ on a Carnot group (when identified in coordinates with some Euclidean space via a diffeomorphism) generate the same topology. Hence the identity map from $(\mathbb{R}^{4},d_{e})$ to $(\mathbb{R}^{4},d_{cc})$ is continuous, where $d_{cc}$ is the CC distance on the Engel group. In particular, it is uniformly continuous on compact sets. 

According to the definition of $\gamma_{c,d_{1},d_{2},d_{3}}$, we can choose a compact set $K\subset \mathbb{R}^{4}$ that contains the image of $\gamma=\gamma_{c,d_{1},d_{2},d_{3}}$ for all  $(c,d_1,d_2,d_3) \in B(0,1)$. Using the uniform continuity described above, given $\varepsilon>0$ there exists $\eta>0$ such that for any $t\in [0,1]$ and $(c,d_1,d_2,d_3) \in B(0,1)$ the following implication holds: $|\gamma(t)-L(t)|<\eta$ implies $d(\gamma(t),L(t))<\varepsilon$. 
Thus it suffices to show there is $\delta>0$ such that $\sup_{t\in [0,1]}|\gamma(t)-L(t)|<\eta$ for all $0<r<\delta$. Since $\gamma(0)=L(0)=0$ we have 
\[ |\gamma(t)-L(t)| = \left| \int_{0}^{t}(\dot{\gamma}(s)-\dot{L}(s)) ds \right| \leq \sup_{s\in [0,1]}   |\dot{\gamma}(s)-\dot{L}(s)|\]
where $\dot{\gamma}=(\dot{\gamma}_{1},\dot{\gamma}_{2},\dot{\gamma}_{3},\dot{\gamma}_{4})$ and $\dot{L}=(\dot{L}_{1},\dot{L}_{2},\dot{L}_{3},\dot{L}_{4})$ are the Euclidean derivatives of the coordinate representations of the curves in $\mathbb{R}^{4}$ as in the previous section. The required fact then follows using the definition of $\dot{L}$ and $\dot{\gamma}$. Indeed, each term $|\dot{\gamma}_{i}(s)-\dot{L}_{i}(s)|$ for $1\leq i\leq 4$ will be a polynomial in $s$ whose coefficients are products of constants and at least one of $c, d_{1}, d_{2}, d_{3}$.

Combining the previous steps and the definition of the distance on $C_{H}^{1}([0,1],\mathbb{E})$, it follows that, $\gamma \in \mathcal{V}$ for all $(c,d_1,d_2,d_3) \in B(0,r)$ when $r>0$ is sufficiently small. This proves the claim and thus the theorem in the case when $a \neq 0$.

\medskip

\noindent \textbf{Case 2.} Suppose $a=0$ and $b\neq 0$.

In this case $V = bX_{2}$ for some $b\neq 0$. Then $\exp(bX_{2})=(0,b,0,0)$. Suppose $\gamma \in C^{1}_H([0,1],\mathbb{E})$ with $\gamma(0)=0$ such that $|(\dot{\gamma}_{1}(t),\dot{\gamma}_{2}(t))-(0,b)|<\frac{|b|}{2}$ for $t\in [0,1]$. Then the sign of $\dot{\gamma}_{2}$ will always be the same as the sign of $b$. Hence the sign of $\gamma_{4}(1)=\frac{1}{2}\int_{0}^{1}\dot{\gamma}_{2}(t)(\gamma_{1}(t))^{2}dt$ is the same as the sign of $b$. Therefore, if $\mathcal{V}$ is an open ball in $C_{H}^{1}([0,1],\mathbb{E})$ with center $t\mapsto \exp(tbX_{2})$ and radius $|b|/2$, the set
\[\{\beta(1):\beta \in \mathcal{V}, (\beta, \beta')(0)=(0,V)\}\]
cannot contain any points $(p_{1},p_{2},p_{3},p_{4})$ where the sign of $p_{4}$ is different from the sign of $b$. In particular, it cannot be a neighborhood of $\exp(bX_{2})=(0,b,0,0)$. Hence $bX_{2}$ is not pliable.
\end{proof}

Theorem \ref{t-engel} implies that every vector in the set $\mathfrak{O} = V_1 \setminus \text{span}\{X_2\}$ is pliable.
Combining this fact with Theorem~\ref{t-partialwhitney} gives the following.

\begin{corollary}\label{EngelWhitney}
    The Engel group has the $C^1_H$ Whitney extension property on $V_1 \setminus \mathrm{span} \{X_2\}$.
\end{corollary}

Combining Corollary \ref{EngelWhitney} with ideas from \cite{JS16}, we prove the following Lusin approximation result in the Engel group. As we were unable to fill in the details of the related proof from \cite{JS16}, we provide an original proof below.

\begin{theorem}\label{engellusin}
Let $\gamma\colon [a,b]\to \mathbb{R}^{4}$ be a horizonal curve in the Engel group. Suppose
\[S:=\{t\in [a,b]:\gamma'(t)\notin \mathrm{Span}(X_{2}) \}\]
satisfies $m(S)>0$. 
Then for any $\varepsilon>0$ there exists a $C^{1}$ horizontal curve $\Gamma\colon [a,b]\to \mathbb{R}^{4}$ such that
\[m\{t\in [a,b]:\Gamma(t)=\gamma(t)\}>m(S)-\varepsilon.\]
\end{theorem}

\begin{proof}


Recall that horizontal curves are, by definition, absolutely continuous. By choosing an arc length parameterization \cite[Lemma 1.1.4]{AGS08}, there is $T>0$, a Lipschitz curve $\phi\colon [0,T]\to \bbG$ (using the CC distance in $\bbG$), and an absolutely continuous, nondecreasing function $F\colon [a,b]\to [0,T]$ such that $\gamma=\phi \circ F$. Let 
\[P:=\{t\in [a,b]:F'(t) \mbox{ exists and is strictly positive}\},\]
\[Z:=\{t\in [a,b]:F'(t)\mbox{ exists and is zero}\}.\]

\medskip

\noindent \textbf{Claim:} The map $\gamma$ is Pansu differentiable almost everywhere with derivative $\gamma'$ given by $\gamma'(t)=0$ for every $t\in Z$ and $\gamma'(t)=\phi'(F(t))F'(t)$ which is horizontal for almost every $t\in P$. 

\medskip

\emph{Proof of Claim:} First suppose $t\in Z$. Then
\[ \frac{d( \phi(F(s)),\phi(F(t)))}{|s-t|}\leq \mathrm{Lip}(\phi)\frac{|F(s)-F(t)|}{|s-t|}\to 0 \mbox{ as }s\to t,\]
which shows $\gamma$ is Pansu differentiable at $t$ with derivative $0$.

We next show $\phi$ is differentiable at $F(t)$ for almost every $t\in P$. Let $D\subset [0,T]$ be the set of points where $\phi$ is not Pansu differentiable or $\phi' \notin V_1$.
Then $m(D)=0$ by Pansu's theorem (Theorem~\ref{pansu}) which also implies that $D$ is Lebesgue measurable. By a change of variables formula \cite[Special Case of 7.2.6 Theorem]{Rud87}, we have
\[ 0 = m(D) = \int_{0}^{T}\chi_{D} = \int_{a}^{b}(\chi_{D}\circ F)F'.\]
It follows that $(\chi_{D}\circ F)F'=0$ almost everywhere in $[a,b]$. Hence $\chi_{D}\circ F=\chi_{F^{-1}(D)}$ is zero at almost every point of $P$, which implies $m(P\cap F^{-1}(D))=0$. That is, $\phi$ is Pansu differentiable at $F(t)$ and $\phi'(F(t)) \in V_1$ for almost every $t\in P$.

Now fix $t\in P$ such that $F$ is differentiable at $t$ and $\phi$ is Pansu differentiable at $F(t)$ with derivative $\phi'(F(t)) \in V_1$. Since also $F$ is continuous, it follows that
\[ \lim_{s\to t} \frac{d(\phi(F(s)),\phi(F(t))\exp((F(s)-F(t))\phi'(F(t))))}{|F(s)-F(t)|}=0.\]
Since $F$ is differentiable at $t$, we obtain
\begin{align*}
 &\lim_{s\to t} \frac{d(\phi(F(s)),\phi(F(t))\exp((F(s)-F(t))\phi'(F(t))))}{|s-t|}\\
 &\qquad =\lim_{s\to t} \frac{d(\phi(F(s)),\phi(F(t))\exp((F(s)-F(t))\phi'(F(t))))}{|F(s)-F(t)|}\frac{|F(s)-F(t)|}{|s-t|}\\
 &\qquad =0.
 \end{align*}
 To prove the claim, it suffices to show
 \begin{equation}
 \label{e-laststep}    
 \frac{d(\phi(F(t))\exp((F(s)-F(t))\phi'(F(t))),\phi(F(t))\exp((s-t)F'(t)\phi'(F(t)))     )}{|s-t|}
 \end{equation}
 vanishes as $s \to t$.
 It is not difficult to show that, for any $V \in V_1$ and $\lambda_1,\lambda_2 \in \mathbb{R}$, we have $d(\exp(\lambda_{1}V),\exp(\lambda_{2} V))=c(V)|\lambda_{1}-\lambda_{2}|$ where $c(V)$ is a real number depending on $V$ but not $\lambda_{1}$ or $\lambda_{2}$. This is indeed the case since $\exp(-\lambda_{1}V)\exp(\lambda_{2}V)=\exp((\lambda_{2}-\lambda_{1})V)$ using the BCH formula, and so (assuming $\lambda_{1}\leq \lambda_{2}$, otherwise a similar argument applies) $t\mapsto \exp(tV)$ for $0\leq t\leq \lambda_{2}-\lambda_{1}$ is a curve joining $0$ to $\exp((\lambda_{2}-\lambda_{1})V)$ with length $c(V)(\lambda_2-\lambda_1)$. The result then follows by combining this fact with the left invariance of $d$ to bound \eqref{e-laststep} by the following:
 \begin{align*}
 &\lim_{s\to t}\frac{d(\exp((F(s)-F(t))\phi'(F(t))),\exp((s-t)F'(t)\phi'(F(t))) )}{|s-t|}\\
 &\qquad =c(\phi'(F(t)))\lim_{s\to t}\frac{|F(s)-F(t)-(s-t)F'(t)|}{|s-t|}=0.
 \end{align*}

\medskip

 \emph{Proof of Theorem from Claim:} For almost every $t\in [a,b]$ we have
 \[ \lim_{s\to t} \frac{d(\gamma(s),\gamma(t)\exp((s-t)\gamma'(t)))}{|s-t|}=0.\]
Since $t\mapsto \gamma'(t)$ is measurable, an elementary measure theory argument gives for any $\varepsilon>0$ a compact set $K\subset S$ with $m(S\setminus K)<\varepsilon$ such that 
\[ \lim_{\eta \downarrow 0} \sup_{\substack{s,t\in K\\ 0<|s-t|<\eta}} \frac{d(\gamma(s),\gamma(t)\exp((s-t)\gamma'(t)))}{|s-t|}=0\]
and $\gamma'$ is continuous on $K$. By definition of $S$, it follows $\gamma'(t)\notin \mathrm{Span}(X_{2})$ for all $t\in K$. Hence the pair $(\gamma|_{K}, \gamma'|_{K})$ satisfies the hypotheses of the $C^{1}_H$ Whitney extension theorem on $\mathfrak{O}:=V_{1}\setminus \mathrm{Span}(X_{2})$ (Definition \ref{directionalWhitney}). Hence by Corollary~\ref{EngelWhitney}, $\gamma|_{K}$ extends to a $C^{1}$ horizontal curve which agrees with $\gamma$ on $K$. Since $m(S\setminus K)<\varepsilon$, this gives the required approximation.
\end{proof}

We are now ready to prove Theorem \ref{answerisno}, our second main result.

\begin{theorem*}[Restatement of Theorem \ref{answerisno}]
Let $\gamma\colon [a,b]\to \mathbb{R}^{4}$ be a horizonal curve in the Engel group. Then there exists a $C^{1}$ horizontal curve $\Gamma\colon [a,b]\to \mathbb{R}^{4}$ such that
\[m\{t\in [a,b]:\Gamma(t)=\gamma(t)\}>0.\]
\end{theorem*}

\begin{proof}
If $m\{t\in [a,b]:\gamma'(t)\notin \mathrm{Span}(X_{2}) \} >0$, then the result follows by Theorem \ref{engellusin} with any choice of sufficiently small $\varepsilon$. 

Suppose this set has zero measure. Then there exists a Borel measurable function $c\colon [a,b]\to \mathbb{R}$ such that $\dot{\gamma}(t)=c(t)X_{2}(\gamma(t))$ for almost every $t\in [a,b]$. Using left translations we may assume $\gamma(a)=0$. Working in second exponential coordinates, by Lemma \ref{Engellift} we have $\dot{\gamma}(t)=(0,c(t),\gamma_{1}(t)c(t),\frac{1}{2}(\gamma_{1}(t))^{2} c(t))$ for almost every $t\in [a,b]$. Examining the first component, $\dot{\gamma}_{1}(t)=0$ for almost every $t\in [a,b]$, so, by absolute continuity, $\gamma_{1}(t)=0$ for every $t\in [a,b]$. A similar argument with the other components yields $\gamma(t)=(0,\gamma_{2}(t),0,0)$ for every $t\in [a,b]$, where $\dot{\gamma}_{2}(t)=c(t)$ for almost every $t\in [a,b]$. 

Focusing on the second component, we have
\[\lim_{s\to t}\frac{\gamma_{2}(s)-\gamma_{2}(t)-(s-t)\dot{\gamma}_{2}(t)}{|s-t|} =0\]
for almost every $t\in [a,b]$. By discarding a set of arbitrarily small Lebesgue measure and using Lusin's theorem and elementary measure theory, we find a compact set $K\subset [a,b]$ of strictly positive measure such that $\gamma_{2}|_{K}$, $\dot{\gamma}_{2}|_{K}$ are uniformly continuous and
\[ \lim_{\delta \downarrow 0} \sup_{ \substack{0<|s-t|<\delta \\ s,t\in K}} \frac{\gamma_{2}(s)-\gamma_{2}(t)-(s-t)\dot{\gamma}_{2}(t)}{|s-t|}=0.\]
These are the hypotheses of the classical Whitney extension theorem for $C^{1}$ extension from $K$. Therefore, there exists a $C^{1}$ function $W\colon [a,b]\to \mathbb{R}$ that agrees with $\gamma_{2}$ on $K$. However, this implies that the curve $t\mapsto \Gamma(t):=(0,W(t),0,0)$ is a $C^{1}$ horizontal curve that agrees with $\gamma$ on $K$. Since $K$ has strictly positive measure, the result follows.
\end{proof}

\section{Appendix}

In this appendix we briefly derive the expressions for the horizontal vector fields in the Engel group $\mathbb{E}$. These were stated in \eqref{engelhorizontalvectors} and used throughout Section \ref{sec-engel}. They seem to be well known in the literature, but we were unable to find an explicit justification.

Recall that if $X,Y$ are elements of the Lie algebra associated to a Lie group, then the Baker Campbell Hausdorff formula (BCH formula) \cite{BLU07} gives an expression for $Z$ such that $\exp(X)\exp(Y)=\exp(Z)$. In a step 3 Carnot group such as the Engel group, the BCH formula reduces to
\[ Z=X+Y+\frac{1}{2}[X,Y]+\frac{1}{12}[X,[X,Y]]-\frac{1}{12}[Y,[X,Y]].\]
Recall the definition of $\mathbb{E}$ from Definition \ref{Engeldef}. In second exponential coordinates, $\mathbb{E}$ is identified with $\mathbb{R}^{4}$ via a diffeomorphism $\psi:\mathbb{E}\to \mathbb{R}^{4}$ defined as follows: if $x=\exp(x_{4}X_{4})\exp(x_{3}X_{3})\exp(x_{2}X_{2})\exp(x_{1}X_{1}) \in \mathbb{E}$, then $\psi(x):=(x_{1},x_{2},x_{3},x_{4})\in \mathbb{R}^{4}$.

For $i=1, 2$ let $Z_{i}:=d\psi(X_{i})$ be the push forward of $X_{i}$ into $T_0 \mathbb{R}^{4}$ 
via $\psi$. We show how to compute the formula for each $Z_{i}$. Fix a point $(x_{1},x_{2},x_{3},x_{4})=\psi(x)\in \bbR^{4}$. Then $Z_{i}(\psi(x))=\frac{d}{dt}\psi(x\exp(tX_{i}))\big|_{t=0}$. Here we recall that $t\mapsto x\exp(tX_{i})$ is a curve through $x$ with derivative $X_{i}(x)$ at $t=0$. To compute this we divide into cases.

If $i=1$, 
\begin{align*}
x\exp(tX_{1})&= \exp(x_{4}X_{4})\exp(x_{3}X_{3})\exp(x_{2}X_{2})\exp(x_{1}X_{1})\exp(tX_{1})\\
&=\exp(x_{4}X_{4})\exp(x_{3}X_{3})\exp(x_{2}X_{2})\exp((x_{1}+t)X_{1}).
\end{align*}
Hence $\psi(x\exp(tX_{1}))=(x_{1}+t,x_{2},x_{3},x_{4})$. Differentiating with respect to $t$ yields $Z_{1}(\psi(x))=(1,0,0,0)$.

If $i=2$, a short calculation with the BCH formula gives
\begin{align}
x\exp(tX_{2})&=(\exp(x_{4}X_{4})\exp(x_{3}X_{3})\exp(x_{2}X_{2}))(\exp(x_{1}X_{1})\exp(tX_{2})) \nonumber \\
& = \exp(x_2X_2+x_3X_3+x_4X_4) \exp(x_1X_1 + x_2X_2 + \tfrac{tx_1}{2}X_3 + \tfrac{tx_1^2}{12}X_4) \nonumber \\
& = \exp \Big( x_{1}X_{1}+(x_{2}+t)X_{2}+(x_{3}+\frac{tx_{1}}{2}-\frac{x_{1}x_{2}}{2})X_{3} \nonumber \\
&\qquad \qquad \qquad \qquad \qquad \qquad \quad  + (x_{4}+\frac{tx_{1}^{2}}{12}+\frac{x_{1}^{2}x_{2}}{12}-\frac{x_{1}x_{3}}{2})X_{4} \Big). \label{fornextstep}
\end{align}
To compute $\psi(x\exp(tX_{2})) = (y_1,y_2,y_3,y_4)$, we must first rewrite the expression for $\exp(tX_{2})$ in the form $\exp(y_{4}X_{4})\exp(y_{3}X_{3})\exp(y_{2}X_{2})\exp(y_{1}X_{1})$. To do so first notice
\begin{align}    
\exp(y_{4}X_{4})&\exp(y_{3}X_{3})\exp(y_{2}X_{2})\exp(y_{1}X_{1}) \nonumber \\
& =  \exp\left( y_{1}X_{1}+y_{2}X_{2}+(y_{3}-\frac{y_{1}y_{2}}{2})X_{3}  + (x_{4}+\frac{y_{1}^{2}y_{2}}{12}-\frac{y_{1}y_{3}}{2})X_{4}\right).\label{nextstep}
\end{align}
This follows immediately from \eqref{fornextstep} by replacing $t$ with $0$ and $x$ with $y$. Since $\exp$ is a diffeomorphism, we can equate coefficients in \eqref{fornextstep} and \eqref{nextstep} to conclude that $y_{1}=x_{1}$, $y_{2}=x_{2}+t$, $y_{3}=x_{3}+tx_{1}$, and $y_{4}=x_{4}+\frac{tx_{1}^{2}}{2}$. Hence $\psi(x\exp(tX_{2}))=(x_{1},x_{2}+t,x_{3}+tx_{1},x_{4}+\frac{tx_{1}^{2}}{2})$. Differentiating gives $Z_{2}(\psi(x))=(0,1,x_{1},\frac{x_{1}^{2}}{2})$, and this verifies \eqref{engelhorizontalvectors}.

\end{document}